\documentclass{amsart}
\usepackage{amsfonts}
\usepackage{amsmath,amssymb}
\usepackage{amsthm}
\usepackage{amscd}
\usepackage{graphics}
\usepackage{graphicx}
{
\theoremstyle{definition}
\newtheorem{Def}{Definition}
\newtheorem{Ex}{Example}

}
\newtheorem{Cor}{Corollary}
\newtheorem{Prop}{Proposition}
\newtheorem{Thm}{Theorem}

\begin{document}
\title[Round fold maps on linear and general bundles]{Round fold maps on manifolds regarded as the total spaces of linear and more general bundles}
\author{Naoki Kitazawa}
\address{Department of Mathematics, Tokyo Institute of Technology, 
2-12-1 Ookayama, Meguro-ku, Tokyo 152-8551, JAPAN, Tel +81-(0)3-5734-2205, 
Fax +81-(0)3-5734-2738,
}
\email{kitazawa.n.aa@m.titech.ac.jp}
\maketitle

\begin{abstract}
({\it Stable}) {\it fold} maps are fundamental tools in studying a generalized theory of the
 theory of Morse functions on
 smooth manifolds and its application to geometry of the manifolds. It is
 important to construct explicit fold
 maps systematically to study smooth manifolds by the theory of fold maps easy to handle. However,
 such constructions have been difficult in general.

\ \ \ {\it Round} fold maps are defined as stable fold maps such that the sets of all the singular values are
 concentric spheres and it was first
 introduced in 2012--2014. The author studied algebraic and differential topological properties
 of such maps and their manifolds and constructed explicit round fold maps. For example, the
 author succeeded in constructing such maps on manifolds regarded as the total spaces of bundles
 over smooth homotopy spheres by noticing that smooth homotopy spheres admit round fold maps
 whose singular sets are connected and more generally, new such maps on manifolds regerded as the total space of circle bundles
 over another manifold admitting a round fold map. In this paper, as an advanced work, we construct new explicit round fold maps on manifolds regarded as the total
 spaces of bundles such that the fibers
 are closed smooth manifolds and that the structure groups are linear and more general bundles over a manifold admitting a round fold map .    
\end{abstract}

\keywords{Singularities of differentiable maps; singular sets, fold maps. \and Differential topology.}
\subjclass{57R45 \and 57N15}

\section{Introduction}
\label{sec:1}
{\it Fold maps} are fundamental tools in studying a generalization of the theory of Morse
 functions and its application to geometry of manifolds.

 A {\it fold map} is defined as a smooth map
 such that each singular point is of the form
$$(x_1, \cdots, x_m) \mapsto (x_1,\cdots,x_{n-1},\sum_{k=n}^{m-i}{x_k}^2-\sum_{k=m-i+1}^{m}{x_k}^2)$$ for two positive integers $m \geq n$ and an integer $0 \leq i \leq m-n+1$ and a Morse function
 is regarded as a fold map, for example.  For a fold map from a closed smooth manifold of dimension $m$ into a smooth
 manifold of dimension $n$ without boundary, the following two hold.
\begin{enumerate}
\item The set of all the singular points (the {\it singular set}) is a closed smooth submanifold of
 dimension $n-1$ of the source manifold. 
\item The restriction map to the singular set is an immersion of codimension $1$.
\end{enumerate}

We also note that if the restriction map to the singular set is an immersion with normal
 crossings, then it is {\it stable} (stable maps are important in the theory of global singularity; see \cite{golubitskyguillemin} for example).

  Constructions of explicit fold maps will help us to
 study smooth manifolds by using the theory of fold maps which are easy to handle and it is very difficult to construct explicit fold maps in general, although
 existence problems for fold maps have been solved under various conditions. However, such fold maps with good properties were
 constructed as we will introduce the following.  

\ \ \ In \cite{burletderham}, \cite{furuyaporto}, \cite{saeki2}, \cite{saekisakuma} and \cite{sakuma}, {\it special
 generic} maps, which are fold maps whose singular points are of the form
$$(x_1, \cdots, x_m) \mapsto (x_1,\cdots,x_{n-1},\sum_{k=n}^{m}{x_k}^2)$$ for two positive integers $m \geq n$, were
 studied. Special generic maps are not so difficult to construct. They were constructed by
 constructing local $C^{\infty}$ maps on manifolds with boundaries and gluing
 them together. For example, by such methods, some special generic
 maps on homotopy spheres including standard spheres are obtained. Furthermore, manifolds admitting special generic
 maps were classified under restrictions on the dimensions of source and
 target manifolds and the fundamental groups of source manifolds. 

\ \ \ Later, in \cite{kitazawa} and \cite{kitazawa2} the
 author introduced {\it round} fold maps, which will be mainly studied in this paper. A {\it round} fold
 map is defined as a fold
 map satisfying the following three.
\begin{enumerate}
\item The singular set is a disjoint union of standard spheres.
\item The restriction to the singular set is an embedding.
\item The set of all the singular values is a disjoint union of spheres embedded concentrically. 
\end{enumerate}

For example, some special generic maps on homotopy spheres are round fold maps whose singular sets are connected (see Example \ref{ex:1} (\ref{ex:1.1}) later and also \cite{saeki2}).

\ \ \ In \cite{kitazawa2}, homology groups and homotopy groups of manifolds admitting round fold maps were studied. Some examples of round fold maps
 and the diffeomorphism types of their source manifolds were
 given by the author in \cite{kitazawa}, \cite{kitazawa3}, \cite{kitazawa4} and \cite{kitazawa5}. For example, we have
 obtained round fold maps on manifolds admitting bundle structures over the $n$-dimensional ($n \geq 2$) standard sphere and manifolds represented as connected
 sums of manifolds admitting bundle structures over the $n$-dimensional ($n \geq 2$) standard sphere whose fibers are diffemorphic to the ($m-n$)-dimensional standard sphere $S^{m-n}$ ($m>n$). In \cite{kitazawa4} and \cite{kitazawa5}, as
 new answers, we have obtained new round fold maps on closed manifolds admitting bundle structures over (exotic) homotopy spheres or
 ones over more general manifolds.

\ \ \ In the last two papers, as a useful tool to construct new round fold maps, a {\it P-operation} has been introduced. Especially, in these papers, a lot of round fold maps from manifolds having
 bundle structures such that the fibers are circles were obtained. In this paper, as a generalized work of \cite{kitazawa4} and \cite{kitazawa5}, we apply P-operations to construct more explicit round fold maps
 on manifolds having bundle structures such that the structure groups are {\it linear} and act on the fibers smoothly. This paper is organized as the following.

\ \ \ In section \ref{sec:2}, we recall {\it round} fold maps and some terminologies
 on round fold maps such as {\it axes} and {\it proper cores}. We also recall a {\it $C^{\infty}$ trivial} round fold map. We introduce results
 on the diffeomorphism types of manifolds admitting $C^{\infty}$ trivial round fold maps shown by the author in \cite{kitazawa} and \cite{kitazawa3}.

\ \ \ In section \ref{sec:3}, we recall a {\it locally $C^{\infty}$ trivial} round fold map, which is a round fold map satisfying
 a kind of triviality around the connected components of the set of singular values. We recall {\it P-operations} defined in \cite{kitazawa4}, which are operations used to construct new round fold maps
 on manifolds having the structures of manifolds admitting bundle structures over manifolds admitting locally $C^{\infty}$ trivial
 round fold maps. More preisely, a P-operation consists of four steps; we decompose the given round fold map, confirm that the restrictions of the bundle over
 the given manifold to the obtained pieces of the source manifold of the round fold map are trivial, construct maps on these pieces and glue these maps together. A construction
 of a round fold map by a P-operation requires us that the bundle is not so complex.  

\ \ \ In section \ref{sec:4}, as main works of the present paper, we apply P-operations
 to construct new round fold maps on manifolds regarded as the total spaces of bundles whose fibers are closed smooth manifolds and whose structure groups are linear and act on the fibers smoothly ({\it linear} bundles) and more general bundles
 over manifolds admitting locally $C^{\infty}$ trivial round
 fold maps. These works
 are regarded as extensions of works \cite{kitazawa3} and \cite{kitazawa5} by the author. In these works, we
 mainly consider bundles whose fibers are circles and we have obtained a lot of new round fold maps and manifolds. In these works, the theory
 of the classification of circle bundles, which is the most fundamental part of
 the theory of characteristic classes of vector bundles discussed in \cite{milnorstasheff}. In the present paper, we consider some appropriate situations and obtain
 new round fold maps and their source manifolds through Theorems \ref{thm:1}-\ref{thm:8} with Examples \ref{ex:2}-\ref{ex:8}. For such general studies, as
 an essential tool, we use more general theory of characteristic classes
 of linear bundles of \cite{milnorstasheff}.  

\ \ \ Throughout this paper, manifolds and maps between manifolds are smooth and of class $C^{\infty}$ unless
 otherwise stated. The base spaces and fibers of bundles in this paper are smooth manifolds and the structure groups of the bundles act
 on the fibers smoothly unless otherwise stated.

\ \ \ Moreover, let $M$ be a closed (smooth) manifold of dimension $m$, let $N$ be a (smooth) manifold
 of dimension $n$ with no boundary, let $f:M \rightarrow N$ be a (smooth) map and let $m \geq n \geq 1$.
 We denote the {\it singular set} of $f$, which is defined as the set consisting of all the singular points of $f$, by $S(f)$. We call
 the set $f(S(f))$ the {\it singular value set} of $f$. We call an inverse image $f^{-1}(p) \in M$ a {\it fiber} of $f$ and if the point
 $p \in N$ is a regular value of $f$, then we call it a {\it regular fiber} of $f$.    

\section{Round fold maps}
\label{sec:2}

In this section, we review round fold maps. See also \cite{kitazawa2}.

\subsection{Terminologies on round fold maps}

\begin{Def}[round fold maps (\cite{kitazawa2})]
\label{def:1}
$f:M \rightarrow {\mathbb{R}}^n$ ($m \geq n \geq 2$) is said to be a {\it round} fold map if $f$ is $C^{\infty}$ equivalent to
 a fold map $f_0:M_0 \rightarrow {\mathbb{R}}^n$ on a closed $C^{\infty}$ manifold $M_0$ such that the following three hold.

\begin{enumerate}
\item The singular set $S(f_0)$ is a disjoint union of ($n-1$)-dimensional standard spheres and consists of $l \in \mathbb{N}$ connected components.
\item The restriction map $f_0 {\mid}_{S(f_0)}$ is an embedding.
\item Let ${D^n}_r:=\{(x_1,\cdots,x_n) \in {\mathbb{R}}^n \mid {\sum}_{k=1}^{n}{x_k}^2 \leq r \}$. Then
 the set $f_0(S(f_0))$ is represented as the disjoint union ${\sqcup}_{k=1}^{l} \partial {D^n}_k$.  
\end{enumerate}

We call $f_0$ a {\it normal form} of $f$. We call a ray $L$ from $0 \in {\mathbb{R}}^n$ an {\it axis} of $f_0$ and
 ${D^n}_{\frac{1}{2}}$ the {\it proper core} of $f_0$. Suppose that for a round fold map $f$, its normal form $f_0$ and diffeomorphisms
 $\Phi:M \rightarrow M_0$ and $\phi:{\mathbb{R}}^n \rightarrow {\mathbb{R}}^n$, $\phi \circ f=f_0 \circ \Phi$. Then
 for an axis $L$ of $f_0$, we also call ${\phi}^{-1}(L)$ an {\it axis} of $f$ and for the proper core ${D^n}_{\frac{1}{2}}$ of $f_0$, we
 also call ${\phi}^{-1}({D^n}_{\frac{1}{2}})$ a {\it proper core} of $f$. 
\end{Def}

Let $f:M \rightarrow {\mathbb{R}}^n$ be a round fold map,let $P$ be a proper core of $f$ and let $L$ be
 an axis of the map $f$. Then, $f^{-1}({\mathbb{R}}^n-{\rm Int} P)$ has a bundle structure over $\partial P$ such that
 the fiber is diffeomorphic to the manifold $f^{-1}(L)$ and that $f {\mid}_{f^{-1}(\partial P)}:f^{-1}(\partial P) \rightarrow \partial P$ defines
 a subbundle of the previous bundle $f^{-1}({\mathbb{R}}^n-{\rm Int} P)$. In this situation, we
 can define a {\it $C^{\infty}$ trivial} round fold map.

\begin{Def}
\label{def:2}
In this situation, a round fold map $f$ is said to be {\it $C^{\infty}$ trivial} if we can take the bundle $f^{-1}({\mathbb{R}}^n-{\rm Int} P)$ as a trivial bundle. 
\end{Def}

We introduce some known examples of round fold maps with their source manifolds.

\begin{Ex}
\label{ex:1}
\begin{enumerate}
\item
\label{ex:1.1}
Special generic maps from $m$-dimensional homotopy spheres into the Euclidean space of dimension $n \geq 2$ ($m \geq n$) such that the singular
 sets are spheres and that the singular value sets are embedded spheres are round fold maps. They are $C^{\infty}$ trivial. The $m$-dimensional standard sphere $S^m$ admits such a round fold map into ${\mathbb{R}}^n$.   
In section 5 of \cite{saeki2}, such a fold map from $m$-dimensional ($2 \leq m <4$, $m \geq 5$) homotopy sphere into ${\mathbb{R}}^2$ is constructed.
\item
\label{ex:1.2}
Let $F \neq \emptyset$ be a closed manifold. Let $M$ be a closed manifold of dimension $m$ regarded as the total space of an $F$-bundle
 over $S^n$ {\rm (}$m \geq n \geq 2${\rm )}. In \cite{kitazawa}, in the case where $F$ is the standard sphere $S^{m-n}$, a round fold map $f:M \rightarrow {\mathbb{R}}^n$ such that the following four hold has been constructed
 and it has been shown that any manifold admitting such a map is regarded as the total space of an $S^{m-n}$-bundle over $S^n$.
\begin{enumerate}
\item $f$ is $C^{\infty}$ trivial.
\item The singular set $S(f)$ has two connected components.
\item For an axis $L$ of $f$, $f^{-1}(L)$ is diffeomorphic to the cylinder $F \times [-1,1]=S^{m-n} \times [-1,1]$.
\item Two connected components of the fiber of a point in a proper core of $f$ is regarded as fibers of the $F$-bundle
 over $S^n$.
\end{enumerate}
In \cite{kitazawa4}, a round fold map $f:M \rightarrow {\mathbb{R}}^n$ satisfying all the conditions but the second condition just before has
 been constructed and it has been shown that any manifold admitting such a map is regarded as the total space of an $F$-bundle over $S^n$. 
\item
\label{ex:1.3}
A special generic map $f$ from an $m$-dimensional closed manifold $M$ into ${\mathbb{R}}^n$ such that the following two hold
 appears in \cite{saeki2}, for example.
\begin{enumerate}
\item The restriction map $f {\mid}_{S(f)}$ is an embedding.
\item $S(f)$ is a disjoint union of two copies of the ($n-1$)-dimensional standard sphere.
\end{enumerate}
For example, the product of $S^{n-1}$ and an ($m-n+1$)-dimensional homotopy sphere admits such a map which is also $C^{\infty}$ trivial.
\end{enumerate}
\end{Ex}

\section{Locally $C^{\infty}$ trivial round fold maps and P-operations}
\label{sec:3}
We recall {\it locally $C^{\infty}$ trivial} round
 fold maps and {\it P-operations}. See also \cite{kitazawa3} and \cite{kitazawa5}.

\begin{Def}
\label{def:3}
Let $f:M \rightarrow {\mathbb{R}}^n$ be a round fold map. Assume that for any connected component $C$ of $f(S(f))$ and a small closed tubular neighborhood $N(C)$ of $C$ such that $\partial N(C)$ is the disjoint
 union of two connected components $C_1$ and $C_2$, $f^{-1}(N(C))$ has the structure
 of a trivial bundle over $C_1$ or $C_2$ and $f {\mid}_{f^{-1}(C_1)}:f^{-1}(C_1) \rightarrow C_1$ and $f {\mid}_{f^{-1}(C_2)}:f^{-1}(C_2) \rightarrow C_2$ give the structures of
 subbundles of the bundle $f^{-1}(N(C))$. Then $f$ is said to be {\it locally $C^{\infty}$ trivial}. We call the fiber $F_C$ of the bundle $f^{-1}(N(C))$ the {\it normal fiber of $C$ corresponding to the bundle $f^{-1}(N(C))$}. Assume that $C_1$ is in the bounded connected
 component of ${\mathbb{R}}^n-C_2$ and we denote the fiber of the subbundle $f^{-1}(C_1)$ by ${\partial}_0 F_C$.
\end{Def}

Maps in Example \ref{ex:1} (\ref{ex:1.1}) are locally $C^{\infty}$ trivial. Ones in Example \ref{ex:1} (\ref{ex:1.2}) are constructed as
 locally $C^{\infty}$ trivial maps in \cite{kitazawa} and \cite{kitazawa4}. We can construct maps on the product of the ($n-1$)-dimensional standard sphere
 and a homotopy sphere explained in the
 last part of Example \ref{ex:1} (\ref{ex:1.3}) as locally $C^{\infty}$ trivial maps. 

We can construct a locally $C^{\infty}$ trivial round fold map in the following method. We use the method in the proof of Proposition \ref{prop:2} and other
 scenes of the present paper.

 Let $l^{\prime} \in \mathbb{N}$. Let $\{E_j\}_{j=1}^{l^{\prime}}$ be a family of compact manifold of dimension $m-n+1$ such that the boudary of $E_j$ is
 a disjoint union of two closed manifolds $F_j$ and $F_{j+1}$, that $F_{l^{\prime}+1}$ is empty and that except $F_0$ and $F_{l^{\prime}+1}$, all the
 manifolds in the family $\{F_j\}$ are non-empty. There exist a positive integer $l$ and a sequence of integers $\{k_j\}_{j=1}^{l^{\prime}}$ of integers such that $k_1=1$ and $k_{l^{\prime}}=l$ hold
 and that the inequality $k_j<k_{j+1}$ holds for any integer $1 \leq j \leq l^{\prime}-1$. We can construct a Morse function $\tilde{f_j}:E_j \rightarrow [k_j-\frac{1}{2},k_{j+1}-\frac{1}{2}]$ satisfying
 the following three.

\begin{enumerate}
\item For any integer $1 \leq j \leq l^{\prime}$, On $F_j$, $\tilde{f_j}$ is constant and minimal if $F_j$ is non-emppty
 and on $F_{j+1}$, $\tilde{f_j}$ is constant and maximal if $F_{j+1}$ is non-empty.  
\item The minimum of $\tilde{f_j}$ is $k_j-\frac{1}{2}$ if $F_j$ is non-empty. If $F_j$ is empty, then by
 the assumption, $j=1$ holds and in this case, the minimum of $\tilde{f_1}$ is $1$. The maximum of $\tilde{f_j}$ is $k_{j+1}-\frac{1}{2}$ if $j \neq l^{\prime}$ holds
 and an integer $l$ larger than $k_{l^{\prime}}$ if $j=l^{\prime}$ holds. respectively. The
 image $\tilde{f_j}({\rm Int} E_j)$ of the interior ${\rm Int} E_j$ of $E_j$ is the
 open interval $(k_j-\frac{1}{2},k_{j+1}-\frac{1}{2})$.
\item Singular points of $\tilde{f_j}$ are always in the interior ${\rm Int} E_j$ of $E_j$ and at distinct singular
 points, the values are always distinct. Furthermore, the set of all the singular values consists
 of all the integers larger than $k_j-\frac{1}{2}$ and smaller than $k_{j+1}-\frac{1}{2}$ if $j \neq l^{\prime}$ holds and all the integers larger
 than $k_j-\frac{1}{2}$ and not larger than $l$ if $j = l^{\prime}$ holds. 
\end{enumerate}

We obtain a family
 of maps $\{\tilde{f_j} \times {\rm id}_{S^{n-1}}:E_j \times S^{n-1} \rightarrow [k_j-\frac{1}{2},k_{j+1}-\frac{1}{2}] \times S^{n-1}\}_{j=1}^{l^{\prime}-1}$. 

\ \ \ If $F_1$ is non-empty, then by gluing the family of maps and the projection $p:{D^n}_{\frac{1}{2}} \times F_1 \rightarrow {D^n}_{\frac{1}{2}}$
 together properly, we obtain a desired round fold map; for a non-negative real number $t$, we
 regard $\{t\} \times S^{n-1}$ as $\partial {D^n}_t$ by identifying $(t,x) \in \{t\} \times S^{n-1}$
 with $(\frac{t}{|x|}x) \in {D^n}_t$. If $F_1$ is empty, then by gluing the family $\{\tilde{f_j} \times {\rm id}_{S^{n-1}}:E_j \times S^{n-1} \rightarrow [k_j-\frac{1}{2},k_{j+1}-\frac{1}{2}] \times S^{n-1}\}_{j=1}^{l^{\prime}-1}$ of
 maps, we obtain a desired round fold map similarly.

We call such a construction a {\it locally trivial spinning construction}.

The following proposition has been shown in \cite{kitazawa} and also in \cite{kitazawa4}.

\begin{Prop}[\cite{kitazawa}]
\label{prop:1}
Let $m,n \in \mathbb{N}$, $n \geq 2$ and $m \geq 2n$.
Any manifold represented as a connected sum of $l \in \mathbb{N}$ closed manifolds regarded as the total spaces of $S^{m-n}$-bundles over $S^n$ admits
 a locally $C^{\infty}$ trivial round fold map $f$ into ${\mathbb{R}}^n$ such that the following four hold.
\begin{enumerate}
\item All the regular fibers of $f$ are disjoint unions of finite copies of $S^{m-n}$.
\item The number of connected components of $S(f)$ and the number of connected components of the fiber of a point in a proper core of $f$ are $l+1$.
\item For any connected component $C$ of $f(S(f))$ and a small closed tubular neighborhood $N(C)$ of $C$, $f^{-1}(N(C))$ is regarded as the total space
 of a trivial bundle over $C$ as in Definition \ref{def:3} such that a normal
 fiber $F_C$ of $C$ corresponding to the bundle $f^{-1}(N(C))$ is diffeomorphic to a disjoint union of a finite number of the following two manifolds.
\begin{enumerate}
\item $D^{m-n+1}$.
\item $S^{m-n+1}$ with the interior
 of a union of disjoint three {\rm (}$m-n+1${\rm )}-dimensional standard
 closed discs removed.
\end{enumerate}
\item All the connected components of the fiber of a point in a proper core of $f$ are regarded as fibers of the $S^{m-n}$-bundles
 over $S^n$ and a fiber of any $S^{m-n}$-bundle over $S^n$ appeared in the connected sum is regarded as a connected component of the
 fiber of a point in a proper core of $f$.  
\end{enumerate}
\end{Prop}

\begin{Prop}[\cite{kitazawa3}]
\label{prop:2}
Let $M$ be a closed manifold of dimension $m$ and $f:M \rightarrow {\mathbb{R}}^n$ be a locally $C^{\infty}$ trivial round
 fold map. Let $F \neq \emptyset$ be a closed manifold and ${M}^{\prime}$ be a closed manifold regarded as the total space of an $F$-manifold over $M$ such
 that for any connected component $C$ of $f(S(f))$ and a small closed tubular neighborhood $N(C)$ of $C$, the
 restriction to $f^{-1}(N(C))$ is a trivial bundle. Then on the manifold ${M}^{\prime}$, there exists a locally $C^{\infty}$ trivial round
 fold map $f^{\prime}:{M}^{\prime} \rightarrow {\mathbb{R}}^n$. 
\end{Prop}

We introduce the proof of this first performed in \cite{kitazawa3}. In the proof of
 this proposition, we use the notation ${D^n}_r:=\{(x_1,\cdots,x_n) \in {\mathbb{R}}^n \mid {\sum}_{k=1}^{n}{x_k}^2 \leq r\}$ appearing in Definition \ref{def:1}.
 
\begin{proof}[Proof of Proposition \ref{prop:2}]

We only prove the proposition in the case where $f$ is locally $C^{\infty}$ trivial, since we can similarly prove this in the case where $f$ is $C^{\infty}$ trivial. We assume that $f(M)$ is diffeomorphic to $D^n$ and we can prove the theorem similarly if $f(M)$ is not diffeomorphic to $D^n$.

\ \ \ We may assume that $f:M \rightarrow {\mathbb{R}}^n$ is a normal form. Let $S(f)$ consist of $l$ connected
 components. Set $P_0:={D^n}_{\frac{1}{2}}$ and $P_k:={D^n}_{k+\frac{1}{2}}-{\rm Int} {D^n}_{k-\frac{1}{2}}$ for an integer $1 \leq k \leq l$. Then ${f}^{-1}(P_k)$ is regarded as the total
 space of a trivial bundle
 over $\partial {D^n}_{k-\frac{1}{2}}$ or $\partial {D^n}_{k+\frac{1}{2}}$ such that the fibers are diffeomorphic to a compact manifold, which
 we denote by $E_k$ and that the submersions $f {\mid}_{{f^{-1}}(\partial {D^n}_{k-\frac{1}{2}})}$ and $f {\mid}_{f^{-1}(\partial {D^n}_{k+\frac{1}{2}})}$ make the
 submanifolds subbundles of the bundle ${f}^{-1}(P_k)$; we denote the fibers of these
 two subbundles by ${E_k}^{1} \subset E_k$ and ${E_k}^2 \subset E_k$, respectively. For any integer $1 \leq k \leq l$ and a diffeomorphism ${\phi}_k$ from
 $f^{-1}(\partial {D^n}_{k-\frac{1}{2}}) \subset f^{-1}(P_k)$ onto $f^{-1}(\partial {D^n}_{k-\frac{1}{2}}) \subset f^{-1}(P_{k-1})$ regarded as a
 bundle isomorphism between the two trivial bundles over standard spheres inducing the identification between the base spaces, $M$ is regarded as $(\cdots ((f^{-1}({D^n}_{\frac{1}{2}})) {\bigcup}_{{\phi}_1} f^{-1}(P_1)) \cdots) {\bigcup}_{{\phi}_l} f^{-1}(P_l)$
 and for any integer $1 \leq k \leq l$ and a diffeomorphism ${\Phi}_k$ from
 $f^{-1}(\partial {D^n}_{k-\frac{1}{2}}) \times F \subset f^{-1}(P_k) \times F$ onto $f^{-1}(\partial {D^n}_{k-\frac{1}{2}}) \times F \subset f^{-1}(P_{k-1}) \times F$ regarded as a
 bundle isomorphism between the two trivial $F$-bundles inducing ${\phi}_k$, ${M}^{\prime}$ is
 regarded as
 $(\cdots ((f^{-1}({D^{n}}_{\frac{1}{2}}) \times F) {\bigcup}_{{\Phi}_1} (f^{-1}(P_1) \times F)) \cdots) {\bigcup}_{{\Phi}_l} ({f^{-1}}(P_l) \times F)$. We construct a map on $f^{-1}(P_k) \times F$. This manifold is regarded as the
 total space of a
 trivial bundle over $\partial {D^n}_{k-\frac{1}{2}}$ or $\partial {D^n}_{k+\frac{1}{2}}$ such that the fibers are diffeomorphic to $E_k \times F$. On $E_k \times F$ there exists
 a Morse function $\tilde{f_k}$ such that the following four hold.

\begin{enumerate}
\item $\tilde{f_k}(E_k \times F) \subset [k-\frac{1}{2},k+\frac{1}{2}]$ and $\tilde{f_k}({\rm Int} (E_k \times F)) \subset (k-\frac{1}{2},k+\frac{1}{2})$ hold. 
\item $\tilde{f_k}({E_k}^1 \times F)=\{k-\frac{1}{2}\}$ holds if ${E_k}^1 \times F$ is non-empty. 
\item $\tilde{f_k}({E_k}^2 \times F)=\{k+\frac{1}{2}\}$ holds if ${E_k}^2 \times F$ is non-empty.
\item Singular points of $\tilde{f_k}$ are in the interior of $E_k \times F$ and at two distinct singular points, the values are always distinct.
\end{enumerate}

We obtain a map ${\rm id}_{S^{n-1}} \times \tilde{f_k}:S^{n-1} \times E_k \times F \rightarrow S^{n-1} \times [k-\frac{1}{2},k+\frac{1}{2}]$. We can identify $S^{n-1} \times [k-\frac{1}{2},k+\frac{1}{2}]$
 with $P_k={D^n}_{k+\frac{1}{2}}-{\rm Int} {D^n}_{k-\frac{1}{2}}$ by identifying $(p,t) \in S^{n-1} \times [k-\frac{1}{2},k+\frac{1}{2}]$ with $tp \in P_k$ where we regard $S^{n-1}$ as the unit sphere of dimension $n-1$. 
 By gluing the composition of the projection from $f^{-1}({D^n}_{\frac{1}{2}}) \times F$ onto $f^{-1}({D^n}_{\frac{1}{2}})$ and $f {\mid}_{f^{-1}({D^n}_{\frac{1}{2}})}:f^{-1}({D^n}_{\frac{1}{2}}) \rightarrow {D^n}_{\frac{1}{2}}$ and the family
 $\{{\rm id}_{S^{n-1}} \times \tilde{f_k}\}$ together by using the family $\{{\Phi}_k \}$ and the family of identifications in the target manifold ${\mathbb{R}}^n$, we obtain
 a new round fold map ${f}^{\prime}:M^{\prime} \rightarrow {\mathbb{R}}^n$. 
\end{proof}

In the proof, from a given map $f:M \rightarrow {\mathbb{R}}^n$, we obtain a new map ${f}^{\prime}:M^{\prime} \rightarrow {\mathbb{R}}^n$. We call the operation of
 constructing ${f}^{\prime}$ from $f$ a {\it P-operation by $F$ to the map $f$}. For example, if $f$ is a map presented in Example \ref{ex:1} (\ref{ex:1.1}), then
 on any manifold having the bundle structure over the source homotopy sphere, we can construct a (locally) $C^{\infty}$ trivial round
 fold map by a P-operation to the map $f$.

\section{Constructions of round fold maps on manifolds regarded as linear bundles by P-operations}
\label{sec:4}

 In this paper, we denote the {\it $k$-th orthogonal group} by $O(k)$ and {\it the $k$-th special orthogonal group} by $SO(k)$. In this paper, a
 bundle is said to be {\it linear} if the structure group is a subgroup of an orthogonal group. A linear bundle is said to be {\it orientable} if the
 structure group is reduced to a subgroup of a special orthogonal group and we obtain two {\it oriented} linear bundles naturally.  

  We recall known fundamental terms and facts on linear bundles.

 For any linear bundle, we can consider its {\it $k$-th Stiefel-Whitney class}, which is a $k$-th cohomology class of the base space whose coefficient ring is $\mathbb{Z}/2\mathbb{Z}$. For any
 oriented linear bundle whose structure group is $SO(k)$, we can consider its {\it Euler class}, which is a $k$-th cohomology class of the base space whose
 coefficient ring is $\mathbb{Z}$. We introduce known facts on classifications of linear bundles without proofs.

\begin{Prop}
\label{prop:3}
Let $X$ be a topological space regarded as a CW-complex.
\begin{enumerate}
\item
\label{prop:3.1}
 The 1st Stiefel-Whitney class $\alpha \in H^1(X;\mathbb{Z}/2\mathbb{Z})$ of a linear bundle over $X$ vanishes if and only if the bundle is orientable.
\item
\label{prop:3.2}
 For any $\alpha \in H^2(X;\mathbb{Z})$, there exists an oriented linear bundle whose structure group is $SO(2)$ whose Euler class is $\alpha$ and the
 2nd Stiefel-Whitney class $\bar{\alpha} \in H^2(X;\mathbb{Z}/2\mathbb{Z})$ of this bundle is the value of the
 canonical homomorphism from $\mathbb{Z}$ onto $\mathbb{Z}/2\mathbb{Z}$, which sends the integer $k$ to $k(\mod 2)$. 
\item
\label{prop:3.3}
 Two oriented linear bundles whose structure groups are $SO(2)$ over $X$ are equivalent if and only if the Euler classes
 are same. Especially, an oriented linear bundle whose structure group is $SO(2)$ over $X$ is trivial if and only if its Euler class vanishes.
\item
\label{prop:3.4}
 Let $k=0, 1, 2, 3$ and let $X$ be simple homotopy equivalent to a $k$-dimensional CW-complex. Let $l$ be an integer larger than $2$. Then, a linear bundle over $X$ such that the structure group is $SO(l)$ and that the 2nd Stiefel-Whitney class vanishes is a trivial linear bundle.
\end{enumerate}
\end{Prop}

 In this section, we define a {\it spin} bundle as an orientable linear bundle such that the 2nd Stiefel-Whitney class vanishes as in Proposition \ref{prop:3} (\ref{prop:3.4}).

 $k$-dimensional real (complex) vector bundles are
 regarded as linear bundles whose structure groups are the groups of all the linear transformations on the
 fibers. In this paper, we only consider real vector bundles. We also note that the structure groups of $k$-dimensional real vector bundles
 are regarded as the groups of all the orthgonal transformations on the $k$-dimensional vector spaces and naturally as $O(k)$. 

 Any $S^{k}$-bundle
 whose structure group consists of linear transformations on the fiber $S^k$ is regarded as a linear bundle whose structure group
 is the group of all the transformations given by the restrictions of orthgonal transformations on ${\mathbb{R}}^{k+1}$ considering $S^k$ as the unit
 sphere and as a result $O(k+1)$. It is naturally a subbundle of an associated real vector bundle whose fiber is a ($k+1$)-dimensional real vector space. We call such a
 linear bundle a {\it standard linear bundle}. For $k=1,2,3$, any $S^k$-bundle
 is regarded as a standard linear bundle whose structure group
 is $O(k+1)$ (see \cite{smale} for the case $k=2$ and \cite{hatcher} for the case $k=3$).
 
  First, we review some topological properties of the total
 spaces of such bundles. For a $k$-dimensional manifold $X$, let
 us denote by $TX$ the total space of the tangent bundle of $X$, which is an important $k$-dimensional
 real vector bundle over $X$ and for $k \geq 2$, let us denote the total space of the unit tangent bundle of $X$ by $UTX$, which
 is obtained as the subbundle of the bundle $TX$ whose fiber is the unit sphere $S^{k-1} \subset {\mathbb{R}}^k$. A manifold is said
 to be {\it spin} if its tangent bundle is spin.

  Also, the following Proposition \ref{prop:4} is useful.

\begin{Prop}
\label{prop:4}

\begin{enumerate}
\item
\label{prop:4.1}
 Let $i \geq 1$ be an integer. For any topological space $E$ regarded as the total space of a standard linear $S^i$-bundle over a topological space $X$ regarded as a CW complex. Let $\pi$ be
 a surjection giving $E$ the bundle structure over $X$. Let $w_{i+1} \in H^{i+1}(X;\mathbb{Z}/2\mathbb{Z})$ be the {\rm (}$i+1${\rm )}-th Stiefel-Whitney class of the bundle $E$ and $\bigcup w_{i+1}$ be the
 operation of taking a cup product with $w_{i+1}$. 
 Then, we have the following exact sequence {\rm (}the
 {\rm Gysin sequence} of the bundle $\pi:E \rightarrow X${\rm )}.

$$
\begin{CD}
@> {\pi}^{\ast}  >> H^j(E;\mathbb{Z}/2\mathbb{Z}) @> >>  H^{j-i}(X;\mathbb{Z}/2\mathbb{Z}) @> \bigcup w_{i+1}  >> H^{j+1}(X;\mathbb{Z}/2\mathbb{Z}) \\ @> {\pi}^{\ast} >> H^{j+1}(E;\mathbb{Z}/2\mathbb{Z}) @> >>  H^{j-i+1}(X;\mathbb{Z}/2\mathbb{Z}) @> \bigcup w_{i+1}  >>
\end{CD}
$$

We also have the following exact sequence in the case where the bundle $E$ is oriented. We denote the Euler class by $e \in H^{i+1}(X;\mathbb{Z})$ and $\bigcup e$ be the operation of taking a cup product with $e$.

$$
\begin{CD}
@> {\pi}^{\ast}  >> H^j(E;\mathbb{Z}) @> >>  H^{j-i}(X;\mathbb{Z}) @> \bigcup e  >> H^{j+1}(X;\mathbb{Z}) \\ @> {\pi}^{\ast} >> H^{j+1}(E;\mathbb{Z}) @> >>  H^{j-i+1}(X;\mathbb{Z}) @> \bigcup e  >>
\end{CD}
$$

\item
\label{prop:4.2}
 Let $X_1$ and $X_2$ be closed manifolds and set $X:=X_1 \times X_2$. Then the tangent bundle $TX$ over $X$ is
 regarded as the Whitney sum of the pull-back of the tangent bundle $TX_1$ over $X_1$ by the canonical
 projection of $X_1 \times X_2$ onto $X_1$ and the pull-back of the tangent bundle $TX_2$ over $X_2$ by the canonical
 projection of $X_1 \times X_2$ onto $X_2$.

\item
\label{prop:4.3}
 Let $X$ be a closed manifold and let $X^{\prime}$ be a manifold regarded as the total space of a standard linear
 bundle whose fiber is diffeomorphic to the standard sphere of dimension $k \geq 1$.
\begin{enumerate}
\item Let $X$ be orientable. If the bundle $X^{\prime}$ is {\rm (}not{\rm )} orientable, then the tangent bundle $TX^{\prime}$ of $X^{\prime}$ is {\rm (}resp. not{\rm )} orientable.
\item Let $k \geq 2$ and let the manifold $X$ be spin. If the bundle $X^{\prime}$ is spin {\rm (}not spin{\rm )}, then the tangent
 bundle $TX^{\prime}$ is spin {\rm (}not spin{\rm )}. 
\end{enumerate}
\end{enumerate}
\end{Prop}

 For the theory of linear bundles and their
 characteristic classes including Stiefel-Whitney classes and Euler classes, see also \cite{milnorstasheff} for example.

 In \cite{kitazawa5}, we have constructed a lot of explicit round
 fold maps on manifolds regarded as the total spaces of $S^1$-bundles over a manifold admitting a locally $C^{\infty}$ trivial round fold map by using P-operations. 
 In this section, we apply P-operations to locally $C^{\infty}$ trivial round fold maps to construct new round fold maps on manifolds regarded as the total spaces
 of linear bundles and more general bundles over the original source manifolds.  

\subsection{Cases for round fold maps between low-dimensional manifolds}

First we show the following theorem, which gives more round fold maps and their source manifolds.

\begin{Thm}
\label{thm:1}
Let $M$ be a closed manifold of dimension $2 \leq m \leq 4$, $f:M \rightarrow {\mathbb{R}}^n$ {\rm (}$m \geq n \geq 2${\rm )} be a locally $C^{\infty}$ trivial round
 fold map. Let $M^{\prime}$ be a manifold regarded as the total space of a linear bundle whose fiber is a closed manifold $F \neq \emptyset$. Then we have the following. 
\begin{enumerate}
\item Let $n=4$. Then by a P-operation by $F$ to $f$, we can obtain a locally $C^{\infty}$ trivial round fold map from $M^{\prime}$ into ${\mathbb{R}}^n= {\mathbb{R}}^4$.
\item Let $(m,n)=(4,3)$ and let $M$ be connected. We assume that for a connected component $C_0$ of the singular value set $f(S(f))$ and a small closed tubular
 neighborhood $N(C_0)$ as in Definition \ref{def:3}, the 2nd Stiefel-Whitney class of
 the restriction of the bundle $M^{\prime}$ above to the image of a section of the trivial ${\partial}_0 F_{C_0}$-bundle vanishes and that for any
 connected component $C$ of the singular value set $f(S(f))$ and a small closed tubular
 neighborhood $N(C)$ as in Definition \ref{def:3}, the restriction of the bundle $M^{\prime}$ above to the normal fiber $F_C$ of $C$ corresponding
 to the bundle $f^{-1}(N(C))$ is orientable. Then
 by a P-operation by $F$ to the map $f$, we can obtain a locally $C^{\infty}$ trivial round fold map from $M^{\prime}$ into ${\mathbb{R}}^n$.     
\item Let $(m,n)=(4,2)$. We assume that for any connected component $C$ of the singular value set $f(S(f))$
 and a small closed tubular
 neighborhood $N(C)$, the restriction of the bundle $M^{\prime}$ to $f^{-1}(N(C))$ is spin. Then
 by a P-operation by $F$ to the map $f$, we can obtain a locally $C^{\infty}$ trivial round fold map from $M^{\prime}$ into ${\mathbb{R}}^n$.
\item Let $(m,n)=(3,3)$ and let $M$ be connected. We assume that for a connected component $C$ of the singular value set $f(S(f))$ and a small closed tubular
 neighborhood $N(C)$ as in Definition \ref{def:3}, the 2nd Stiefel-Whitney class of
 the restriction of the bundle $M^{\prime}$ to the image of a section of the trivial ${\partial}_0 F_{C_0}$-bundle vanishes. Then
 by a P-operation by $F$ to the map $f$, we can obtain a locally $C^{\infty}$ trivial round fold map from $M^{\prime}$ into ${\mathbb{R}}^n$. 
\item Let $(m,n)=(3,2)$. We assume that for any connected component $C$ of the singular value set $f(S(f))$
 and a small closed tubular
 neighborhood $N(C)$, the restriction of the bundle $M^{\prime}$ to $f^{-1}(N(C))$ is spin. Then
 by a P-operation by $F$ to the map $f$, we can obtain a locally $C^{\infty}$ trivial round fold map from $M^{\prime}$ into ${\mathbb{R}}^n$. 
\item Let $(m,n)=(2,2)$ and let $M$ be connected. We assume that for a connected component $C_0$ of the singular value set $f(S(f))$ and a small closed tubular
 neighborhood $N(C_0)$ as in Definition \ref{def:3}, the restriction of the bundle $M^{\prime}$ to the image of a section  of the trivial ${\partial}_0 F_{C_0}$-bundle is orientable. Then
 by a P-operation by $F$ to the map $f$, we can obtain a locally $C^{\infty}$ trivial round fold map from $M^{\prime}$ into ${\mathbb{R}}^n$. 
\end{enumerate}
\end{Thm}
\begin{proof}
By virtue of Proposition \ref{prop:2}, to prove the statements, it suffices to show that for any connected component $C$ of the singular value set $f(S(f))$
 and a small closed tubular neighborhood $N(C)$ as in Definition \ref{def:3}, the restriction of the bundle $M^{\prime}$ over $M$ to $f^{-1}(N(C))$ is trivial. Note that $f^{-1}(N(C))$
 is regarded as a trivial bundle whose fiber is $F_C$ as menitioned in Definition \ref{def:3}. Note also that $F_C$ is regarded as an ($m-n+1$)-dimensional CW-complex
 simple homotopy equivalent to an ($m-n$)-dimensional CW-complex and that $f^{-1}(N(C))$ is regarded as an $m$-dimenisonal CW-complex simple
 homotopy equivalent to the product of $S^{n-1}$ and the ($m-n$)-dimensional CW-complex before.

We prove the first case. $n=4$ or $(m,n)=(4,4)$ is assumed. In this case, $f^{-1}(N(C))$ is a compact $4$-dimensional manifold diffeomorphic to the
 product of $S^3$ and the closed interval and regarded as a CW-complex
 simple homotopy equivalent to $S^3$. We have $H^1(f^{-1}(N(C));\mathbb{Z}/2\mathbb{Z}) \cong H^2(f^{-1}(N(C));\mathbb{Z}) \cong \{0\}$
 and from Proposition \ref{prop:3} (\ref{prop:3.4}), the restriction of the bundle $M^{\prime}$ over $M$ to $f^{-1}(N(C))$ is trivial.

We prove the second case. $(m,n)=(4,3)$ is assumed. Thus, $f^{-1}(N(C))$ is a compact $4$-dimensional manifold and regarded as a CW-complex
 simple homotopy equivalent to a $3$-dimensional CW-complex and as the product of the $2$-dimensional sphere $S^2$ and the
 compact surface $F_C$ with non-empty boundary. It is
 assumed that for a connected
 component $C_0$ of the singular value set $f(S(f))$ and a small closed tubular
 neighborhood $N(C_0)$, the 2nd Stiefel-Whitney class of
 the restriction of the bundle above to $f^{-1}(N(C_0))$ vanishes and $M$ is connected. Moreover, it is assumed that for any
 connected component $C$ of the singular value set $f(S(f))$ and a small closed tubular
 neighborhood $N(C)$ as in Definition \ref{def:3}, the restriction of the bundle $M^{\prime}$ above to the normal fiber $F_C$ of $C$ corresponding
 to the bundle $f^{-1}(N(C))$ is orientable. We have $H^1(f^{-1}(N(C));\mathbb{Z}/2\mathbb{Z}) \cong H^1(F_C;\mathbb{Z}/2\mathbb{Z})$
 and $H^2(f^{-1}(N(C));\mathbb{Z}/2\mathbb{Z}) \cong H^2(C;\mathbb{Z}/2\mathbb{Z}) \oplus H^2(F_C;\mathbb{Z}/2\mathbb{Z}) \cong H^2(C;\mathbb{Z}/2\mathbb{Z}) \cong \mathbb{Z}/2\mathbb{Z}$ for any
 connected component $C$ of the singular value set $f(S(f))$. By these facts, for any connected
 component $C$ of the singular value set $f(S(f))$ and a small closed tubular
 neighborhood $N(C)$, the 2nd Stiefel-Whitney class of
 the restriction of the bundle $M^{\prime}$ to $f^{-1}(N(C))$ vanishes. From Proposition \ref{prop:3} (\ref{prop:3.4}), the restriction of
 the bundle $M^{\prime}$ over $M$ to $f^{-1}(N(C))$ is spin and trivial. 

We prove the third and fifth cases. In each case, the result follows from the assumption that for any connected component $C$ of the singular value set $f(S(f))$ and a
 small closed tubular neighborhood $N(C)$ as in Definition \ref{def:3}, the restriction of the bundle $M^{\prime}$
 to $f^{-1}(N(C))$, which is regarded as a CW-complex simple homotopy equivalent to a CW-complex of dimension $2$ or $3$, is
 a spin bundle together with Proposition \ref{prop:3} (\ref{prop:3.4}). 

We can prove the fourth and sixth cases by applying methods similar to that of the second case. So we omit the proof.
\end{proof}

 By considering specific cases of some cases of Theorem \ref{thm:1}, as a corollary, we have the following.

\begin{Cor}
\label{cor:1}
Let $M$ be a closed manifold of dimension $2 \leq m \leq 4$, $f:M \rightarrow {\mathbb{R}}^n$ {\rm (}$m \geq n \geq 2${\rm )} be a locally $C^{\infty}$ trivial round
 fold map. Let $M^{\prime}$ be a manifold regarded as the total space of a linear bundle whose fiber is a closed manifold $F \neq \emptyset$. Then
 we have the following.
\begin{enumerate}
\item
\label{cor:1.1}
 Let $(m,n)=(3,2), (4,2), (4,3)$. If the bundle above is a spin bundle, then
 by a P-operation by $F$ to the map $f$, we have a locally $C^{\infty}$ trivial round fold map from $M^{\prime}$ into ${\mathbb{R}}^n$. 
\item
\label{cor:1.2}
 Let $(m,n)=(3,3)$. If the 2nd Stiefel-Whitney class of the bundle above vanishes, then
 by a P-operation by $F$ to the map $f$, we have a locally $C^{\infty}$ trivial round fold map from $M^{\prime}$ into ${\mathbb{R}}^n$.   
\item
\label{cor:1.3}
 Let $(m,n)=(2,2)$. If the bundle above is an orientable bundle, then
 by a P-operation by $F$ to the map $f$, we have a locally $C^{\infty}$ trivial round fold map from $M^{\prime}$ into ${\mathbb{R}}^n$. 
\end{enumerate}
\end{Cor}

\begin{Ex}
\label{ex:2}
\begin{enumerate}
\item
\label{ex:2.1}
In the situation of the former part of Example \ref{ex:1} (\ref{ex:1.2}), let $(m,n)=(4,2)$ and the given map $f$ be $C^{\infty}$ trivial
 and locally $C^{\infty}$ trivial or in the situation of Proposition \ref{prop:1}, let $(m,n)=(4,2)$ and $l=1$. 

 The source manifold $M$ in the example is $S^2 \times S^2$ or a manifold regarded as the total space of a non-trivial
 $S^2$-bundle over $S^2$, which is not spin. For both cases, we have $H^2(M;\mathbb{Z}) \cong {\mathbb{Z}}^2$. Let $M=S^2 \times S^2$, which is
 naturally regarded as the total space of a trivial $S^2$-bundle over $S^2$
 and let $\alpha,\beta \in H^2(M;\mathbb{Z}) \cong {\mathbb{Z}}^2$ be generaters represented by the base space
 and the fiber of the trivial bundle, respectively. Let $a,b \in {\mathbb{Z}}$ and consider the Whitney sum of
 two real (oriented) vector bundles of dimension $2$ whose Euler classes are $a\alpha$ and $2b\beta$, respectively. We immediately have the subbundle whose fiber is the
 unit sphere $S^3$ (we denote the total space by $M^{\prime}$) and the Euler class of the bundle is $2ab$ times a generator of
 the cohomology group $H^4(M;\mathbb{Z}) \cong \mathbb{Z}$ and by Proposition \ref{prop:4} (\ref{prop:4.1}), we have $H^4(M^{\prime};\mathbb{Z}) \cong \mathbb{Z}/|2ab|\mathbb{Z}$. More precisely, we determine
 this cohomology group by using the following Gysin sequence where $\bigcup e$ is the operation of taking a cup product with the Euler class $e \in H^{4}(M;\mathbb{Z})$ of the standard linear bundle $M^{\prime}$ and we often use similar methods in this paper.

$$
\begin{CD}
@>  >> H^3(M^{\prime};\mathbb{Z}) @> >>  H^{0}(M;\mathbb{Z}) \cong \mathbb{Z} @> \bigcup e  >> H^{4}(M;\mathbb{Z}) \cong \mathbb{Z} \\ @> >> H^{4}(M^{\prime};\mathbb{Z}) @> >>  H^{1}(M;\mathbb{Z}) \cong \{0\} @> \bigcup e  >>
\end{CD}
$$

 If $a$ is even, then it
 is a spin bundle and if $a$ is odd, then
 it is not a spin bundle. If the bundle is spin, then
 on the total space $M^{\prime}$ of the $S^3$-bundle over $M=S^2 \times S^2$, we
 can construct a round fold map by applying Theorem \ref{thm:1} or Corollary \ref{cor:1} (\ref{cor:1.1}) and the resulting source manifold is spin (we can apply
 both Theorem \ref{thm:1} and Corollary \ref{cor:1} (\ref{cor:1.1})). Even if the bundle is not spin, then on the total space $M^{\prime}$ of the $S^3$-bundle over $M=S^2 \times S^2$, we
 can construct a round fold map by applying Theorem \ref{thm:1} (we cannot apply Corollary \ref{cor:1} (\ref{cor:1.1})) and the resulting source manifold is not
 spin. More generally, let $a_1, a_2, b_1, b_2 \in \mathbb{Z}$ and consider the Whitney sum of
 two (oriented) real vector bundles of dimension $2$ whose Euler classes are $a_1\alpha+2b_1\beta$ and $a_2\alpha+ 2b_2\beta$, respectively. We immediately
 have the subbundle whose fiber is the
 unit sphere $S^3$ (we denote the total space by $M^{\prime}$), the Euler class of the bundle is $2a_1b_2+2a_2b_1$ times a generator of
 the cohomology group $H^4(M;\mathbb{Z}) \cong \mathbb{Z}$ and we have $H^4(M^{\prime};\mathbb{Z}) \cong \mathbb{Z}/|2a_1b_2+2a_2b_1|\mathbb{Z}$ by applying Proposition \ref{prop:4} (\ref{prop:4.1}). Moreover, we
 can construct a round fold map from $M^{\prime}$ into ${\mathbb{R}}^n$ by applying Theorem \ref{thm:1} or Corollary \ref{cor:1} (\ref{cor:1.1}).

Furthermore, let $M$ be the total space of a non-trivial
 $S^2$-bundle over $S^2$ in this situation, or more generally, in the situation of Proposition \ref{prop:1}, let $(m,n)=(4,2)$. By applying Theorem \ref{thm:1}, we can
 obtain explicit locally $C^{\infty}$ trivial round fold maps similarly. For the latter case, see also Example \ref{ex:6} (\ref{ex:6.1}) later.

\item
\label{ex:2.2}
In the situation of Example \ref{ex:1} {\rm (}\ref{ex:1.3}{\rm )}, let $(m,n)=(4,3)$. The source
 manifold $M:=S^2 \times S^2$ of dimenison $m=4$ admits a $C^{\infty}$ trivial and locally $C^{\infty}$ trivial round fold map
 into ${\mathbb{R}}^n={\mathbb{R}}^3$. We
 define $\alpha,\beta \in H^2(M;\mathbb{Z}) \cong {\mathbb{Z}}^2$ as in the example just
 before. Let $a_1, a_2, b_1, b_2 \in \mathbb{Z}$ and let the sum $a_1+a_2$ be even. Consider the Whitney sum of
 two (oriented) real vector bundles of dimension $2$ whose Euler classes are $a_1\alpha+b_1\beta$ and $a_2\alpha+b_2\beta$, respectively. Similarly, we immediately
 have the subbundle whose fiber is the
 unit sphere $S^3$ (we denote the total space by $M^{\prime}$), the Euler class of the bundle is $a_1b_2+a_2b_1$ times a generator of
 the cohomology group $H^4(M;\mathbb{Z}) \cong \mathbb{Z}$ and we have $H^4(M^{\prime};\mathbb{Z}) \cong \mathbb{Z}/|a_1b_2+a_2b_1|\mathbb{Z}$ by applying Proposition \ref{prop:4} (\ref{prop:4.1}). Moreover, we can construct
 a round fold map from $M^{\prime}$ into ${\mathbb{R}}^n$ by applying Theorem \ref{thm:1} or Corollary \ref{cor:1} (\ref{cor:1.1}).
\end{enumerate}
\end{Ex}

\subsection{Cases for round fold maps such that regular fibers are disjoint unions of spheres}

  In the previous subsection, we obtained new round fold maps by applying P-operations to some locally $C^{\infty}$ trivial round fold maps
 from $m$-dimensional round fold maps into ${\mathbb{R}}^n$ under the constraint that $4 \geq m \geq n \geq 2$ holds. As specific cases, we applied
 P-operations to locally $C^{\infty}$ trivial round fold maps in the former part of Example \ref{ex:1} (\ref{ex:1.2}) and (\ref{ex:1.3}) and Proposition \ref{prop:1}. Here, for
 general pairs $(m,n)$ of dimensions, we apply P-operations to round fold maps satisfying these conditions.

\begin{Thm}
\label{thm:2}
Let $m,n \in \mathbb{N}$. Let $m> n \geq 2$ hold. Let an $m$-dimensional connected closed
 manifold $M$ admit a locally $C^{\infty}$ trivial round fold map into ${\mathbb{R}}^n$ satisfying the
 following conditions as mentioned in Proposition \ref{prop:1}.

\begin{enumerate}
\item All the regular fibers of $f$ are disjoint unions of finite copies of $S^{m-n}$.
\item The number of connected components of $S(f)$ and the number of connected components of the fiber of a point in a proper core of $f$ coincide.
\item For any connected component $C$ of $f(S(f))$ and a small closed tubular neighborhood $N(C)$ of $C$, $f^{-1}(N(C))$ is regarded as the total space of
 a trivial bundle over $C$ as in Definition \ref{def:3} such that a normal
 fiber $F_C$ of $C$ corresponding to the bundle $f^{-1}(N(C))$ is homeomorphic to a disjoint union of a finite number of the following manifolds.
\begin{enumerate}
\item $D^{m-n+1}$.
\item $S^{m-n+1}$ with the interior
 of a union of disjoint three {\rm (}$m-n+1${\rm )}-dimensional standard
 closed discs removed.
\end{enumerate}
Thus, on any manifold $M^{\prime}$ regarded as the total space of a bundle over $M$ such that the restriction to any connected component of the fiber of a point in a proper core of $f$ is trivial, by
 a P-operation, we have a round fold map into ${\mathbb{R}}^n$.
\end{enumerate}

\end{Thm} 

\begin{proof}
We consider a normal form $f_0:M \rightarrow {\mathbb{R}}^n$ of the map $f$. If we restrict the bundle $M^{\prime}$ over $M$ to the inverse image $f^{-1}({D^n}_{\frac{1}{2}})$ of
 the proper core ${D^n}_{\frac{1}{2}}$, then it is trivial by the assumption on $M^{\prime}$.

   For any connected component $C$ of the singular value set $f(S(f))$, we denote a small closed tubular neighborhood as in Definition \ref{def:3} by $N(C)$
 and the normal fiber corresponding to the trivial bundle $f^{-1}(N(C))$ explained in Definition \ref{def:3} by $F_C$. In the definition, the subbundle
 of the bundle $f^{-1}(N(C))$ whose fiber is ${\partial}_0 F_C \subset F_C$ is defined and we assume that the restriction
 of the bundle $M^{\prime}$ over $M$ to the total space of this subbundle is trivial. Then, by the assumed conditions, the restriction of the bundle
 to $f^{-1}(N(C))$ is also trivial. More precisely, we have this fact as the following. 

  By the third condition, the fiber of any connected component of the bundle $f^{-1}(N(C))$ is
 $D^{m-n+1}$ or $S^{m-n+1}$ with the interior
 of a union of disjoint three {\rm (}$m-n+1${\rm )}-dimensional standard
 closed discs removed. By considering the intersection of the fiber of each connected component
 of the bundle $f^{-1}(N(C))$ and the fiber $F_C$ of the bundle $f^{-1}(N(C))$, we obtain a subbundle of each connected component
 of the bundle $f^{-1}(N(C))$. By the second condition, the fiber of the resulting bundle is homeomorphic to the sphere $S^{m-n}$ if the fiber of the connected
 component of $f^{-1}(N(C))$ is homeomorphic to $D^{m-n+1}$ and the fiber of the resulting bundle is homeomorphic to the disjoint union of two copies of the sphere $S^{m-n}$ if the fiber of the connected
 component of $f^{-1}(N(C))$ is homeomorphic to $S^{m-n+1}$ with the interior
 of a union of disjoint three {\rm (}$m-n+1${\rm )}-dimensional standard
 closed discs removed. By considering the homotopy types of the fibers of the connected components of the trivial bundle $f^{-1}(N(C))$, we have the desired fact.     

   By the induction, if we restrict the bundle $M^{\prime}$ over $M$ to $f^{-1}(N(C))$ for any connected component $C$ of the singular value set $f(S(f))$, then it is trivial. Thus we have the statement. 

\end{proof}

\begin{Ex}
\label{ex:3}
Let $m$ and $n$ be integers satisfying the relation $m>n \geq 2$. Let $M$ be a manifold regarded
 as the total space of a trivial $S^{m-n}$-bundle over $S^n$ and let $f:M \rightarrow {\mathbb{R}}^n$ be a
 locally $C^{\infty}$ trivial round fold map presented in Example \ref{ex:1} (\ref{ex:1.2}).
\begin{enumerate}
\item
\label{ex:3.1}
 We consider the Whitney sum of the pull-back
 of a trivial $k_1$-dimensional real vector bundle over $S^{m-n}$ by
 the canonical projection of the
 product $M=S^{m-n} \times S^n$ onto $S^{m-n}$ and the pull-back of a $k_2$-dimensional real vector bundle over $S^n$ by
 the canonical projection of the
 product $M=S^{m-n} \times S^n$ onto $S^n$. Let $M^{\prime}$ be the total space of the subbundle of the real vector bundle whose fiber is the ($k_1+k_2-1$)-dimensional unit
 sphere. Then, the bundle $M^{\prime}$ is trivial over the fiber of a point in a proper core of $f$. We can apply Theorem \ref{thm:2} to obtain a round fold map from $M^{\prime}$ into ${\mathbb{R}}^n$.

  In this situation, let $n=2$ hold.
 In addition, let $k_2=2$ hold. For any integer $k$, we can take the mentioned $2$-dimensional (oriented)
 real vector bundle over $S^n=S^2$ as a bundle
 whose Euler class is $k$ times a generator
 of the group $H^2(S^2;\mathbb{Z}) \cong \mathbb{Z}$. Furthermore, if $k$ is even (odd), then the resulting bundle $M^{\prime}$ is spin (resp. not spin) and
 the manifold $M^{\prime}$ is spin (resp. not spin). Let $k_2 >2$ hold. We consider the mentioned $k_2$-dimensional
 real vector bundle over $S^n=S^2$. Thus, if the mentioned $k_2$-dimensional
 real vector bundle is spin (not spin), then the resulting bundle $M^{\prime}$ is spin (resp. not spin) and
 the manifold $M^{\prime}$ is spin (resp. not spin). 
\item
\label{ex:3.2}
 We consider the Whitney sum of the pull-back
 of the tangent bundle over $S^{m-n}$ by the canonical projection of the
 product $M=S^{m-n} \times S^n$ onto $S^{m-n}$ and the pull-back of a $k$-dimensional real vector bundle over $S^n$ by the canonical projection of the
 product $M=S^{m-n} \times S^n$ onto $S^n$. Let $M^{\prime}$ be the total space of the subbundle of the real vetor bundle whose fiber is the ($m-n+k-1$)-dimensional unit
 sphere. Then, the bundle $M^{\prime}$ is trivial over the fiber of a point in a proper core of $f$. We can apply Thorem \ref{thm:2} to obtain a round fold map from $M^{\prime}$ into ${\mathbb{R}}^n$.

In this situation, let $m$ and $n$ be even and let $k=n$. Then, for a generator $\alpha$ of the cohomology group $H^m(M;\mathbb{Z}) \cong \mathbb{Z}$, we can construct the (oriented) bundles $M^{\prime}$ over $M$ whose Euler
 classes are $0$ and $4\alpha$. In fact, we can take the mentioned (oriented) $n$-dimensional real vector bundle over $S^n$ as a trivial bundle and also as a tangent
 bundle over $S^n$. Especially, if $n=2$ holds, then for any integer $l$, we can construct the bundles $M^{\prime}$ over $M$ whose Euler
 class is $2l\alpha$. In fact, we can take the mentioned oriented $2$-dimensional real vector bundle over $S^n=S^2$ as a bundle whose Euler class is $l$ times a generator
 of the group $H^2(S^2;\mathbb{Z}) \cong \mathbb{Z}$.
 
\item
\label{ex:3.3}
 Let $m \geq 4$ and let $n=m-2 \geq 2$ in this situation. We consider the Whitney sum of the pull-back
 of a $k_1$-dimensional vector bundle over $S^2$ which is spin by the canonical projection of the
 product $M=S^{m-n} \times S^n=S^{m-2} \times S^2$ onto $S^{m-n}=S^{m-2}$ and the pull-back of a $k_2$-dimensional real vector
 bundle over $S^{m-2}$ by the canonical projection of the
 product $M=S^{2} \times S^{m-2}$ onto $S^{m-2}$. Let $M^{\prime}$ be the total space of the subbundle of the
 real vector bundle whose fiber is the ($k_1+k_2-1$)-dimensional unit
 sphere. Then, the bundle $M^{\prime}$ is trivial over the fiber of a point in a proper core of $f$. We can apply Theorem \ref{thm:2} to obtain a round fold map from $M^{\prime}$ into ${\mathbb{R}}^n$.

As a specific case, let $n=m-2=4$ or $m=6$ and let $k_1=2$ and $k_2=m-2=6-2=4$. For any integer $k$, we can take a $k_1$-dimensional (oriented) real
 vector bundle whose Euler class is $k$ times a generator of the cohomology group $H^{2}(S^{2};\mathbb{Z}) \cong \mathbb{Z}$ of
 the base space $S^2$. For any integer $k$, we can take a $k_2$-dimensional (oriented) real
 vector bundle whose Euler class is $k$ times a generator of the cohomology group $H^{m-2}(S^{m-2};\mathbb{Z}) \cong H^{4}(S^4;\mathbb{Z}) \cong \mathbb{Z}$ of
 the base space $S^{m-2}=S^4$. We consider the Whitney sum of the pull-backs defined before.
 The resulting real vector bundle is of dimension $k_1+k_2=2+4=6$ and for any integer $k$, we can obtain
 this vector bundle so that the Euler class is $2k$ times a generator of the cohomology group $H^{m}(M;\mathbb{Z}) \cong H^{m}(S^2 \times S^{m-2};\mathbb{Z}) \cong \mathbb{Z}$.
 It follows that for any integer $k$, we can obtain the total space $M^{\prime}$ of the subbundle of the
 real vector bundle whose fiber is the $5$-dimensional unit sphere satisfying $H^6(M^{\prime};\mathbb{Z}) \cong \mathbb{Z}/|2k|\mathbb{Z}$.
\end{enumerate}
\end{Ex}

\begin{Thm}
\label{thm:3}

Let $m,n \in \mathbb{N}$. Let $m>n \geq 2$ hold. Let an $m$-dimensional connected closed
 manifold $M$ admit a locally $C^{\infty}$ trivial round fold map $f$ into ${\mathbb{R}}^n$ which
 is special generic and whose image is diffeomorphic to the cylinder $S^{n-1} \times [-1,1]$ as presented in Example \ref{ex:1} {\rm (}\ref{ex:1.3}{\rm )}. Thus, on
 any manifold $M^{\prime}$ regarded as the total space of a bundle over $M$ such that for any {\rm (}$n-1${\rm )}-dimensional standard sphere $C^{\prime}$ embedded
 in the interior of the image $f(M)$, the restriction to the image of a section of
 the trivial bundle given by $f {\mid}_{f^{-1}(C^{\prime})}:f^{-1}(C^{\prime}) \rightarrow C^{\prime}$ is trivial, by
 a P-operation to the original map $f$, we have a round fold map into ${\mathbb{R}}^n$.

\end{Thm}

\begin{proof}
Let $C$ be a connected component of the singular value set $f(S(f))$ and we take a
 small closed tubular neighborhood $N(C)$ as presented in Definition \ref{def:3}. The obtained bundle $f^{-1}(N(C))$ is a trivial bundle and the normal
 fiber $F_C$ corresponding to the bundle $f^{-1}(N(C))$ is the ($m-n+1$)-dimensional standard closed disc $D^{m-n+1}$. By the extra
 assumption on the bundle $M^{\prime}$ over $M$ that for any {\rm (}$n-1${\rm )}-dimensional standard sphere $C^{\prime}$ embedded
 in the interior of the image $f(M)$, the restriction to the image of a section of
 the trivial bundle given by $f {\mid}_{f^{-1}(C^{\prime})}:f^{-1}(C^{\prime}) \rightarrow C^{\prime}$ is trivial, the fact that the previous trivial bundle
 is a subbundle of the bundle $f^{-1}(N(C))$ 
 and the fact that the fiber of the trivial bundle $f^{-1}(N(C))$ is diffeomorphic to the disc $D^{m-n+1}$ and contractible, if we restrict
 the bundle $M^{\prime}$ to $f^{-1}(N(C))$, then it is trivial. This completes the proof. 
\end{proof}

\begin{Ex}
\label{ex:4}
Let $m$ and $n$ be integers satisfying the relation $m>n \geq 2$ as assumed in Example \ref{ex:3}. Let $\Sigma$ be
 an ($m-n+1$)-dimensional homotopy sphere and let $M=\Sigma \times S^{n-1}$.
 We consider a round fold map $f:M \rightarrow {\mathbb{R}}^n$ presented in Example \ref{ex:1} (\ref{ex:1.3}) or in the assumption of Theorem \ref{thm:3}.
\begin{enumerate}
\item
\label{ex:4.1}
   We consider
 the Whitney sum of the pull-back
 of a $k_1$-dimensional real vector bundle over $\Sigma$ by the canonical projection of the product $M=\Sigma \times S^{n-1}$ onto $\Sigma$ and the pull-back of a trivial $k_2$-dimensional
 real vector bundle over $S^{n-1}$ by
 the canonical projection of the
 product $M=\Sigma \times S^{n-1}$ onto $S^{n-1}$. Let $M^{\prime}$ be the total space of the subbundle of the real vector bundle whose fiber
 is the ($k_1+k_2-1$)-dimensional unit
 sphere. Then, the bundle $M^{\prime}$ is trivial over the image of the section of the bundle $f^{-1}({C}^{\prime})$
 in Theorem \ref{thm:3} and we can apply Theorem \ref{thm:3}.

\item
\label{ex:4.2}
 We consider the Whitney sum of the pull-back
 of the tangent bundle of $S^{n-1}$ by the canonical projection of the
 product $M=\Sigma \times S^{n-1}$ onto $S^{n-1}$ and the pull-back of a $k$-dimensional real vector bundle over $\Sigma$ by the canonical projection of the
 product $M=\Sigma \times S^{n-1}$ onto $\Sigma$. Let $M^{\prime}$ be the total space
 of the subbundle of the real vector bundle whose fiber is the ($n+k-1$)-dimensional unit
 sphere. Then, the bundle $M^{\prime}$ is trivial over the image of the section of the bundle $f^{-1}({C}^{\prime})$ in Theorem \ref{thm:3} and we can apply Theorem \ref{thm:3}.

In this situation, let $m$ and $n-1$ be even. Then, for a generator $\alpha$ of the homology group $H^m(M;\mathbb{Z}) \cong \mathbb{Z}$, we can construct the bundles $M^{\prime}$ over $M$ whose Euler
 classes are $0$ and $4\alpha$. In fact, we can take the mentioned $k$-dimensional (oriented) real vector bundle over $\Sigma$ as a trivial
 bundle of dimension $k=m-n+1$ and also as a tangent
 bundle of $\Sigma$, which is of dimension $k=m-n+1$. Especially, if $m-(n-1)=2$ or $m-n=1$ holds, then for any integer $l$, we can construct the bundle $M^{\prime}$ over $M$ whose Euler
 class is $2l\alpha$. In fact, we can take the mentioned $k$-dimensional (oriented) real vector bundle over $S^2$ as a real vector bundle of
 dimension $k=2$ whose Euler class is $l$ times a generator
 of the group $H^2(S^2;\mathbb{Z}) \cong \mathbb{Z}$. In addition, if $l$ is odd (even), then the bundle $M^{\prime}$ is spin (resp. not spin) and the manifold $M^{\prime}$ is spin (resp. not spin).  

\item
\label{ex:4.3}
 Let $n=3$. We consider the Whitney sum of the pull-back of a $k_1$-dimensional vector bundle over $S^{n-1}=S^2$ which is spin by the canonical projection of the
 product $M=S^{m-2} \times S^2$ onto $S^2$ and the pull-back of a $k_2$-dimensional real vector
 bundle over $S^{m-2}$ by the canonical projection of the
 product $M=S^{m-2} \times S^2$ onto $S^{m-2}$. Let $M^{\prime}$ be the total space of the subbundle
 of the real vector bundle whose fiber is the ($k_1+k_2-1$)-dimensional unit
 sphere. Then, the bundle $M^{\prime}$ is trivial over the image of the section of the bundle $f^{-1}({C}^{\prime})$ in Theorem \ref{thm:3} and we can apply Theorem \ref{thm:3}.
\end{enumerate}
\end{Ex} 

\subsection{Cases for round fold maps into the plane}

  In \cite{kitazawa5}, we have obtained a lot of round fold maps by P-operations by the circle $S^1$ to a locally $C^{\infty}$ trivial round
 fold map into the plane. In this paper, we construct such maps on manifolds regarded as the total spaces of (more general) linear bundles by P-operations.

 We introduce a class of round fold maps first introduced in \cite{kitazawa2}.

\begin{Def}
\label{def:4}
Let $f:M \rightarrow {\mathbb{R}}^n$ be a round fold map, and let $R$ be a commutative group.

Let $P$ be a proper core of $f$. Then, $f^{-1}({\mathbb{R}}^n-{\rm Int} P)$ has a bundle structure mentioned
 just before Definition \ref{def:2}. 
$f$ is said to be {\it homologically $R$-trivial} if for a bundle $f^{-1}({\mathbb{R}}^n-{\rm Int} P)$, the following diagram commutes for
 the canonical projection $p:\partial P \times f^{-1}(L) \rightarrow \partial P$, the projection
 of the bundle $\pi:f^{-1}({\mathbb{R}}^n-{\rm Int} P) \rightarrow \partial P$ and two isomorphisms of
 homology groups $\Phi$ and $\phi$ for any integer $j$. 

$$
\begin{CD}
H_j({E}^{\prime};R) @> \Phi  >>H_j(\partial P \times f^{-1}(L);R) \\
@VV {\pi}_{\ast} V @VV p_{\ast} V \\
H_j(\partial P;R) @> \phi >> H_j(\partial P;R) \\
\end{CD}
$$

\end{Def}

  We have the following.

\begin{Thm}
\label{thm:4}
Let $M$ be a closed manifold of
 dimension $m \geq 2$ and let $f:M \rightarrow {\mathbb{R}}^2$ be a locally $C^{\infty}$ trivial and homologically $\mathbb{Z}$-trivial round fold map
 such that for any connected component $C$ of the singular
 value set $f(S(f))$, a small closed tubular neighborhood $N(C)$ and $F_C$ as
 in Definition \ref{def:3}, $H^1(F_C;\mathbb{Z}) \cong H^2(F_C;\mathbb{Z}) \cong \{0\}$ holds. 
Thus, we
 have a family of manifolds $\{M_K\}_{K \in H_1(f^{-1}(L);\mathbb{Z}/2\mathbb{Z}) \oplus H_2(f^{-1}(L);\mathbb{Z}/2\mathbb{Z})}$ regarded
 as the total spaces of linear bundles over $M$ and a family of
 round fold maps $\{f_K:M_K \rightarrow {\mathbb{R}}^n\}_{K \in H_1(f^{-1}(L);\mathbb{Z}/2\mathbb{Z}) \oplus H_2(f^{-1}(L);\mathbb{Z}/2\mathbb{Z})}$.

 Furthermore, we have the following three statements.

\begin{enumerate}
\item
\label{thm:4.1}
 We have the linear bundles $M_K$ as standard linear bundles whose fibers are the standard sphere $S^k$ of dimension $k>1$ and we can construct
 the maps $f_K$ so that they are homologically $\mathbb{Z}$-trivial.
\item
\label{thm:4.2}
 Let the manifold $f^{-1}(L)$ be spin and for a proper core $P$ of $f$, the manifold $f^{-1}({\mathbb{R}}^n-{\rm Int} P)$
 be not spin. In this case, $f$ is not $C^{\infty}$ trivial. If the linear bundles $M_K$ are standard linear bundles
 whose fibers are diffeomorphic to the standard sphere $S^k$ of dimension $k>1$, then for an element
 $K \in H_1(f^{-1}(L);\mathbb{Z}/2\mathbb{Z}) \oplus H_2(f^{-1}(L);\mathbb{Z}/2\mathbb{Z})$, the map $f_K$
 is not $C^{\infty}$ trivial.
\item
\label{thm:4.3}
 Let the manifolds $f^{-1}(L)$ and $f^{-1}({\mathbb{R}}^n-{\rm Int} P)$ be spin. In this case, we can
 construct the map $f_K$ so that it is not $C^{\infty}$ trivial for any
 element $K=(c,0) \in H_1(f^{-1}(L);\mathbb{Z}/2\mathbb{Z}) \oplus H_2(f^{-1}(L);\mathbb{Z}/2\mathbb{Z})$ where $c \in H_1(f^{-1}(L);\mathbb{Z}/2\mathbb{Z})$ is not zero.  
\end{enumerate}
\end{Thm}
\begin{proof}

From the assumption that $H^1(F_C;\mathbb{Z}) \cong H^2(F_C;\mathbb{Z}) \cong \{0\}$ holds, we
 have $H^2(f^{-1}(N(C)) \times F_C;\mathbb{Z}) \cong \{0\}$. Here, we consider the situation of Definition \ref{def:4} and
 abuse notation there. Let $k_0 \in H_1(\partial P;R) \cong \mathbb{Z}$ be a generator of the group, For $K \in H_1(f^{-1}(L);\mathbb{Z}/2\mathbb{Z})$, let us regard the tensor product $k_0 \otimes K$ 
 as an element of $H_2(\partial P \times f^{-1}(L);R)$ by considering the natural identification and we denote the 2nd homology
 class ${\Phi}^{-1}(k_0 \otimes K) \in H_2(f^{-1}({\mathbb{R}}^n-{\rm Int} P);\mathbb{Z}/2\mathbb{Z})$ by $K^{\prime}$.
 
\ \ \ By applying Proposition \ref{prop:3} (\ref{prop:3.2}), we can obtain a manifold $M_K$ regarded as the total space
 of a linear bundle whose structure group is $SO(2)$ such that the 2nd Stiefel-Whitney class is the dual of $K^{\prime} \in H_2(f^{-1}({\mathbb{R}}^n-{\rm Int} P);\mathbb{Z}/2\mathbb{Z})$ and construct
 the desired round fold map $f_K:M_K \rightarrow {\mathbb{R}}^n$. By constructing the manifold $M_K$ as the total space of a standard linear bundle whose fiber
 is diffeomorphic to $S^k$ satisfying $k>1$, we
 easily have the first statement of the latter three statements too. The former part of the second statement of the latter three statements is clear. In the
 situation of this statement, there exist a class $K$ and the corresponding manifold $M_K$ regarded as the total space of a linear bundle
 whose structure group is $SO(2)$ such that the 2nd Stiefel-Whitney class of the bundle is the dual of $K^{\prime} \in H_2(f^{-1}({\mathbb{R}}^n-{\rm Int} P);\mathbb{Z}/2\mathbb{Z})$
 and that this 2nd Stifel-Whitney class and that of the tangent bundle of the manifold $f^{-1}({\mathbb{R}}^n-{\rm Int} P)$ do not
 coincide (set $K=0$ for example). Thus, we have a round fold map $f_K:M_K \rightarrow {\mathbb{R}}^n$ which is not $C^{\infty}$ trivial since the total space of the bundle obtained by
 the restriction of the
 bundle $M_K$ to $f^{-1}({\mathbb{R}}^n-{\rm Int} P)$ is not spin by Proposition \ref{prop:4} (\ref{prop:4.3}). Last, in the situation of the
 last statement of the three statements, the resulting manifold $M_K$ is not spin, for an proper core $P_K$ of the resulting round fold map $f_K$, the inverse
 image $f^{-1}({\mathbb{R}}^n-{\rm Int} P_K)$ is not spin, and for an axis $L_K$ of the resulting map $f_K:M_K \rightarrow {\mathbb{R}}^n$,
 the inverse image ${f_k}^{-1}(L_K)$ is spin. Thus, we obtain the last statement of the latter three.
\end{proof}

\begin{Ex}
\label{ex:5}
\begin{enumerate}
\item
\label{ex:5.1}
 In the situation of the explanation of a locally trivial spinning construction
 introduced after Definition \ref{def:3}, let $n=2$ and $l^{\prime}=3$. Moreover, let $E_1$ and $E_2$ be a manifold homeomorphic to the standard sphere of dimension $k \geq 4$ with the
 interior of disjoint three smoothly embedded $k$-dimensional standard closed discs removed; moreover, let $F_2$ be
 the disjoint union of two copies of a ($k-1$)-dimensional homotopy sphere and let $F_1$ and $F_3$ be a ($k-1$)-dimensional homotopy sphere and let
 $E_3$ be a $k$-dimensional standard closed disc. By performing the construction, we have a round fold map satisfying
 the assumption of Theorem \ref{thm:4}. 
\item
\label{ex:5.2}
 In the situation of Example \ref{ex:1} or Theorem \ref{thm:3}, let $m=n=2$. Then $M$ is the torus $S^1 \times S^1$. We may apply
 Theorem \ref{thm:3} or \ref{thm:4} to obtain a round fold map from a manifold $M^{\prime}$ regarded as the total space of a standard linear bundle
 over the torus $M$ whose fiber is diffeomorphic to the standard sphere of dimension $k \geq 3$; especially, we can take the linear bundle and the manifold $M^{\prime}$ that are not spin. Furthermore, we
 can obtain a resulting round fold map satisfying the assumption of the previous exmaple by a P-operation. This resulting map satisfies the assumption of Theorem \ref{thm:4} (\ref{thm:4.2}).
\end{enumerate}
\end{Ex}

\subsection{Other cases}

First, by applying Proposition \ref{prop:4} (\ref{prop:4.2}), we easily have the following proposition.

\begin{Prop}
\label{prop:5}
Let $X_1$ and $X_2$ be topological spaces and let ${\pi}_i:X_1 \times X_2 \rightarrow X_i$ be the canonical
 projection {\rm (}$i=1,2${\rm )}. Let $B_i$ be regarded as the total space of a real
 vector bundle over $X_i$. Assume also that the following two hold.
\begin{enumerate}
\item The vector bundle $B_1$ over $X_1$ is trivial.
\item The Whitney sum of the vector bundle $B_2$ over $X_2$ and a trivial real vector bundle over $X_2$ of dimension not larger than
 that of the vector bundle $B_1$ is a trivial vector bundle over $X_2$.
\end{enumerate}
Then, the Whitney sum of the vector bundle over $X_1 \times X_2$ defined as the pull-back of the bundle $B_1$ by the projection ${\pi}_1$ and
 the bundle defined as the pull-back of the bundle $B_2$ by the projection ${\pi}_2$ is a trivial vector bundle over $X_1 \times X_2$.  
\end{Prop}

By virtue of Propositions \ref{prop:2} and \ref{prop:5}, we immediately have the following.

\begin{Prop}
\label{prop:6}
Let $M$ be a closed manifold of dimension $m \geq 2$, let $f:M \rightarrow {\mathbb{R}}^n$ {\rm (}$m \geq n \geq 2${\rm )} be a locally $C^{\infty}$ trivial round
 fold map. Let $M^{\prime}$ be a manifold regarded as the total space of a standard linear bundle over $M$ such that
 the following two hold.
\begin{enumerate}
\item For any connected component $C$ of the singular value set $f(S(f))$ and a small closed tubular
 neighborhood $N(C)$ as in Definition \ref{def:3}, the restriction of the bundle $M^{\prime}$ to $f^{-1}(N(C))$, which is regarded as
 the total space of a trivial bundle over $C$, is equivalent to the Whitney sum of the following two
 real vector bundles $E_1$ and $E_2$, where $F_C$ is
 the normal fiber corresponding to
 the trivial bundle $f^{-1}(N(C))$.
\begin{enumerate}
\item The pull-back $E_1$ of a real vector bundle over $C$ by the projection of the trivial bundle $f^{-1}(N(C))$ over $C$.
\item The pull-back $E_2$ of a real vector bundle
 over $F_C$ by the canonical projection of $f^{-1}(N(C))$, regarded as $C \times F_C$, onto $F_C$. 
\end{enumerate}
\item One of the previous two bundles $E_1$ {\rm (}$E_2${\rm )} is trivial and the Whitney sum of the other bundle $E_2$ {\rm (}resp. $E_1${\rm )} and
 a trivial real vector bundle of dimension
 not larger than that of the trivial real vector bundle $E_1$ {\rm (}resp. $E_2${\rm )} is trivial.  
\end{enumerate}
 Then, by a
 P-operation to the map $f$, we can construct a locally $C^{\infty}$ trivial round fold map $f^{\prime}:M^{\prime} \rightarrow {\mathbb{R}}^n$.
\end{Prop}

We have the following theorem.

\begin{Thm}
\label{thm:5}
Let $m,n \in \mathbb{N}$ and $m \geq n \geq 2$. Let $M$ be a closed manifold of dimension $m$ and let $f:M \rightarrow {\mathbb{R}}^n$ be a locally $C^{\infty}$ trivial round
 fold map. Let $M^{\prime}$ be a manifold regarded as the total space of a standard linear bundle over $M$ such that
 for any connected component $C$ of the singular value set $f(S(f))$ and a small closed tubular neighborhood $N(C)$, the restriction
 of the bundle $M^{\prime}$ over $M$ to $f^{-1}(N(C))$ as in Definition \ref{def:3} is equivalent to the unit tangent bundle $UTf^{-1}(N(C))$ of $f^{-1}(N(C))$.
Assume that either of the following two holds.
\begin{enumerate}
\item $n=2,4,8$ and the Whitney sum of the tangent bundle of the normal fiber $F_C$ corresponding to the bundle $f^{-1}(N(C))$ and a trivial real vector bundle of dimension $n-1$ is trivial. 
\item The tangent bundle of the normal fiber $F_C$ corresponding to the bundle $f^{-1}(N(C))$ is trivial. In other word, $F_C$ is parallelizable.
\end{enumerate}
Then, by a P-operation by $S^{m-1}$ to the map $f$, we can obtain a locally $C^{\infty}$ trivial round fold map $f^{\prime}:M^{\prime} \rightarrow {\mathbb{R}}^n$.
\end{Thm}

\begin{proof}
For any positive integer $k$, the tangent bundle $TS^k$ of the sphere $S^k$ is stably parallelizable, or the Whitney sum of the bundle $TS^k$ and a trivial real vector
 bundle of dimension $1$ over $S^k$ is trivial. Moreover, for $k=1,3,7$, the tangent bundle is trivial. By virtue of Proposition \ref{prop:6}, in the situation
 of this theorem, the tangent
 bundle $Tf^{-1}(N(C))$ and the unit tangent bundle $UTf^{-1}(N(C))$ is trivial.
 From Proposition \ref{prop:2}, we have a round fold map $f^{\prime}:M^{\prime} \rightarrow {\mathbb{R}}^n$ by applying a P-operation.
\end{proof}

\begin{Ex}
\label{ex:6}
\begin{enumerate}
\item
\label{ex:6.1}
 We can apply Theorem \ref{thm:5} to the map $f:M \rightarrow {\mathbb{R}}^n$ in Proposition \ref{prop:1} to
 construct a round fold map on the total space $UTM$ of the unit tangent bundle of $M$. In
 the situation of Proposition \ref{prop:1}, let the integers $m$ and $n$ be even. In this situation, the Euler
 class of the tangent bundle $TM$ and the unit tangent bundle $UTM$ of $M$ is $4l$ times a generator of
 the cohomology group $H^{m}(M;\mathbb{Z}) \cong \mathbb{Z}$ and we have $H^{m}(UTM;\mathbb{Z}) \cong \mathbb{Z}/4l\mathbb{Z}$. 
In the case where $(m,n)=(4,2)$ is assumed, we can obtain such maps also by applying Theorem \ref{thm:1}. See also Example \ref{ex:2} (\ref{ex:2.1}). 
\item
\label{ex:6.2}
 Let $m,n \in \mathbb{N}$ and let $n \geq 2$. Let $M$ be a closed manifold of dimension $m$ admitting a locally $C^{\infty}$ trivial round
 fold map $f:M \rightarrow {\mathbb{R}}^n$. Assume that one of the following three holds.
\begin{enumerate}
\item $m=n,n+1$.
\item $m=n+2$ and for any connected component $C$ and a small closed tubular neighborhood $N(C)$ as in Definition \ref{def:3}, the tangent bundle of the normal fiber $F_C$ corresponding to the bundle $f^{-1}(N(C))$ is orientable.
\item $m=n+3$ and for any connected component $C$ and a small closed tubular neighborhood $N(C)$ as in Definition \ref{def:3}, the tangent bundle
 of the normal fiber $F_C$ corresponding to the bundle $f^{-1}(N(C))$ is spin.
\end{enumerate}
In this case, for any connected component $C$ and a small closed tubular neighborhood $N(C)$ as in Definition \ref{def:3}, the tangent bundle
 of the normal fiber $F_C$ corresponding to the bundle $f^{-1}(N(C))$ is always trivial. We can apply Theorem \ref{thm:5} to
 the map $f:M \rightarrow {\mathbb{R}}^n$ to
 construct a round fold map on the total space $UTM$ of the unit tangent bundle of $M$.
\end{enumerate}
\end{Ex}

We also have the following.

\begin{Thm}
\label{thm:6}
Let $M$ be a closed manifold of dimension $2 \leq m \leq 4$, let $f:M \rightarrow {\mathbb{R}}^n$ {\rm (}$m \geq n \geq 2${\rm )} be
 a locally $C^{\infty}$ trivial round
 fold map. Let $M^{\prime}$ be a manifold regarded as the total space of the subbundle of a normal bundle obtained by considering
 an immersion of the manifold $M$ into an Euclidean space of codimension $k>2$ whose fiber is the unit sphere
 of dimension $k-1$.
Assume that for any connected component $C$ and a small closed tubular neighborhood $N(C)$ as in Definition \ref{def:3}, the normal fiber $F_C$ corresponding
 to the bundle $f^{-1}(N(C))$ is orientable. Then, on the manifold $M^{\prime}$, by
 a P-operation to the original map $f$, we can obtain a round fold map into ${\mathbb{R}}^n$.
\end{Thm}

\begin{proof}
As discussed in Example \ref{ex:6} (\ref{ex:6.2}), the tangent bundle
 of the normal fiber $F_C$ corresponding to the bundle $f^{-1}(N(C))$ is always trivial and the tangent bundle $Tf^{-1}(N(C))$ of the
 manifold $f^{-1}(N(C))$ is also trivial by virtue of Proposition \ref{prop:6}. From this fact, the restriction of the bundle $M^{\prime}$
 over $M$ to $f^{-1}(N(C))$ is orientable and spin. By Proposition \ref{prop:3} (\ref{prop:3.4}), the obtained bundle
 over $f^{-1}(N(C))$ is trivial. We can apply Proposition \ref{prop:2} and this completes the proof.
\end{proof}

Last, we prove two theorems. Before that, we define a {\it trivial} embedding of a standard sphere into a manifold as
 an embedding of the sphere into the interior of the latter manifold which is smoothly isotopic to an unknot in the interior of a standadrd closed
 disc embedded in the interior of the manifold.  

\begin{Thm}
\label{thm:7}
Let $m,n \in \mathbb{N}$ satisfying $m \geq n \geq 2$. Let $M$ be a closed connected manifold of dimension $m$ and let $f:M \rightarrow {\mathbb{R}}^n$ be
 a locally $C^{\infty}$ trivial round
 fold map. Let $P$ be a proper core of $f$ and
 we define
 a compact manifold $\bar{M}$ of dimension $m$ as the union
 of $f^{-1}({\mathbb{R}}^n-{\rm Int} P)$ and some connected components of the manifold $f^{-1}(P)$. Assume also that $f(\bar{M})$ is
 diffeomorphic to $D^n$; in other words, $f(M)$ must be diffeomorphic to $D^n$ and $\bar{M}$ must include at least
 one connected component of $f^{-1}(P)$. 
 
Let $M^{\prime}$ be
 a manifold regarded as the total space of a bundle over $M$ whose fiber is
 a closed connected manifold $F \neq \emptyset$ such that the restriction of the bundle to the previous
 manifold $\bar{M}$ is a trivial bundle. Assume also that {\rm (}the embedding of{\rm )} the inverse image $f^{-1}(\partial f(M))$ of the
 boundary $\partial f(M)$ is a trivial embedding into $\bar{M}$. Then by a P-operation by $F$ to the map $f$, we can obtain a round
 fold map $f^{\prime}:M^{\prime} \rightarrow {\mathbb{R}}^n$ such that {\rm (}the embedding of{\rm )} the inverse image ${f^{\prime}}^{-1}(\partial f^{\prime}(M^{\prime}))$ of the
 boundary $\partial f^{\prime}(M^{\prime})$ is a trivial embedding into $M^{\prime}$. 
\end{Thm}

\begin{proof}
From Proposition \ref{prop:2} and the assumption that the restriction of the bundle $M$ to $\bar{M} \supset f^{-1}({\mathbb{R}}^n-{\rm Int} P)$ is
 trivial, we can construct a locally $C^{\infty}$ trivial round fold map $f^{\prime}:M^{\prime} \rightarrow {\mathbb{R}}^n$ by a P-operation by $F$ to the map $f$. To show that we can construct such a map satisfying
 the additional property, we study the structure of obtained map by noticing the definition of a P-operation or the construction demonstrated
 in the proof of Proposition \ref{prop:2}. We abuse notation and terminologies in the proof of Proposition \ref{prop:2}.
 
We may regard the given map $f$ as a normal form. By the definition of a P-operation and the mentioned bundle structure of the restriction of the bundle $M^{\prime}$ over $M$ to $\bar{M}$, in the proof
 of Proposition \ref{prop:2}, we can choose the bundle isomorphism ${\Phi}_k$ ($k \neq 1$) as the product of the identification
 map ${\phi}_k$ from $f^{-1}(\partial {D^n}_{k-\frac{1}{2}}) \subset f^{-1}(P_k)$
 onto $f^{-1}(\partial {D^n}_{k-\frac{1}{2}}) \subset f^{-1}(P_{k-1})$ and the identity map ${\rm id}_{F}$. We can
 choose the bundle isomorphism ${\Phi}_1$ so that its restriction to the restriction of
 the bundle $M^{\prime}$ over $M$ to $(\partial f^{-1}({D^n}_{\frac{1}{2}})) \bigcap {\rm Int} \bar{M} \subset \partial f^{-1}({D^n}_{\frac{1}{2}}) \subset f^{-1}(P_1)$ 
 is the product of the identification
 map between the resulting base spaces and the identity map ${\rm id}_{F}$.
Moreover, (the embedding of) the inverse image $f^{-1}(\partial f(M))$ of the
 boundary $\partial f(M)$ into $\bar{M}$ is assumed to be trivial. By virtue of these facts, we can
 construct the map $f^{\prime}$ so that {\rm (}the embedding of{\rm )} the
 inverse image ${f^{\prime}}^{-1}(\partial f^{\prime}(M))$ of the
 boundary $\partial f^{\prime}(M^{\prime})$ into $\bar{M} \times F$ is smoothly isotopic to the restriction of the section of the
 trivial bundle $\bar{M} \times F$ over $\bar{M}$ to $f^{-1}(\partial f(M)) \subset \bar{M}$. The image of this restriction
 map is regarded as $f^{-1}(\partial f(M)) \times \{p\} \subset \bar{M} \times F$
 where $p$ is a point in $F$. Hence, {\rm (}the embedding of{\rm )} the
 inverse image ${f^{\prime}}^{-1}(\partial f^{\prime}(M))$ of the
 boundary $\partial f^{\prime}(M^{\prime})$ into $\bar{M} \times F$ is a trivial embedding into the total
 space $\bar{M} \times F$ of a resulting trivial bundle
 and $M^{\prime}$. This completes the proof.
\end{proof}

\begin{Ex}
\label{ex:7}
Maps in Proposition \ref{prop:1} satisfy the assumption of Theorem \ref{thm:7}. In \cite[EXAMPLE 2]{kitazawa3}, by P-operations by $S^1$ to a map $f$ in
 Proposition \ref{prop:1} the author
 obtained a lot of round fold maps and source manifolds under the assumption that $n=2$ and $m-n \geq 3$ hold (see also \cite[THEOREM 4]{kitazawa3}). By virtue of Theorem \ref{thm:7}, by such a P-operation by $S^1$ to the map $f$, we can obtain a round
 fold map $f^{\prime}:M^{\prime} \rightarrow {\mathbb{R}}^n$ such that {\rm (}the embedding of{\rm )} the
 inverse image ${f^{\prime}}^{-1}(\partial f^{\prime}(M^{\prime}))$ of the
 boundary $\partial f^{\prime}(M^{\prime})$ is a trivial embedding into $M^{\prime}$. More generally, if $M^{\prime}$ is
 regarded as the total space of a bundle whose structure group is $SO(2)$, we can perform the construction similarly.    
\end{Ex}

On the other hand, we also have the following theorem.

\begin{Thm}
\label{thm:8}
Let $m,n \in \mathbb{N}$ satisfying $m \geq n \geq 2$. Let $M$ be a closed connected manifold of dimension $m$ and let $f:M \rightarrow {\mathbb{R}}^n$ be
 a locally $C^{\infty}$ trivial round
 fold map. Let $C_0$ be the connected component of the boundary $\partial f(M)$ of the image $f(M)$ bounding
 the unbounded connected component of the set ${\mathbb{R}}^n-{\rm Int} f(M)$ and assume also that {\rm (}the embedding of{\rm )} the inverse
 image $f^{-1}(C_0)$ of the
 component $C_0$ into $M$ is not null-homotopic. 

 Furthermore, let $F \neq \emptyset$ be a connected manifold such that the group ${\pi}_{n-1}(F)$ is zero and let $M^{\prime}$ be a
 manifold regarded as the total space of an $F$-bundle such that for any connected component $C$ of the singular
 value set $f(S(f))$ and a small closed tubular neighborhood $N(C)$ as in Definition \ref{def:3}, the restriction of the bundle to the
 space $f^{-1}(N(C))$ is trivial and that the homomorphism from ${\pi}_{n-2}(F)$ into ${\pi}_{n-2}(M^{\prime})$ induced by the natural inclusion $i$ is injective.

Then by a P-operation by $F$ to the map $f$, we can obtain a round
 fold map $f^{\prime}:M^{\prime} \rightarrow {\mathbb{R}}^n$. Furthermore, for the connected component ${C_0}^{\prime}$ of the boundary $\partial f^{\prime}(M^{\prime})$ of
 the image $f(M^{\prime})$ bounding
 the unbounded connected component of the set ${\mathbb{R}}^n-{\rm Int} f^{\prime}(M^{\prime})$ such that {\rm (}the embedding of{\rm )} the
 inverse image ${f^{\prime}}^{-1}({C_0}^{\prime})$ of the
 component ${C_0}^{\prime}$ into $M^{\prime}$ is also null-homotopic. 
\end{Thm}
\begin{proof}
We have the following homotopy exact sequence

$$
\begin{CD}
@>   >> {\pi}_{n-1}(F) \cong \{0\} @> >>  {\pi}_{n-1}(M^{\prime}) @>  >> {\pi}_{n-1}(M) \\ @> >> {\pi}_{n-2}(F) @> {i}_{\ast} >> {\pi}^{n-2}(M^{\prime}) @> >>. 
\end{CD}
$$

Since the last homomorphism is assumed to be injective, The homomorphism ${\pi}_{n-1}(M^{\prime};\mathbb{Z})$ into ${\pi}_{n-1}(M)$ is an isomorphism.

From this, we immediately have the result.

\end{proof}

\begin{Ex}
\label{ex:8}
We review the construction of the map presented in the former part of Example \ref{ex:1} (\ref{ex:1.2}) done in \cite{kitazawa} and \cite{kitazawa4} in
 the case where the source manifold $M$ is a manifold regarded as the total space of an $S^1$-bundle over $S^2$ which is not homeomorphic to $S^3$. 
 In these proofs, essentially, P-operations are used. 

More precisely, in the proof of Proposition \ref{prop:2}, set $l=1$ or consider a P-operation
 to a map presented in Example \ref{ex:1} (\ref{ex:1.1}) from $S^2$ into the plane. In the last step, we need to take the
 diffeomorphism ${\Phi}_1$ used in the proof appropriately. ${\Phi}_1$ is regarded as a bundle isomorphism between two
 trivial $S^1 \sqcup S^1$-bundles over $S^1$. By considering the structure of the $S^1$-bundle over $S^2$ and well-known
 facts on the diffeomorphism group of $S^1 \times S^1$, we can construct a round fold map $f:M \rightarrow {\mathbb{R}}^2$ so
 that the inverse image $f^{-1}(\partial f(M))$ of the boundary $\partial f(M)$ of the image $f(M)$ is smoothly
 isotopic to the fiber of a point in a proper core of $f$ and that the fiber of the point is also regarded as
 a fiber of the $S^1$-bundle $M$ over $S^2$.

As a result, we have a round fold map $f:M \rightarrow {\mathbb{R}}^2$ satisfying the assumption
 of Theorem \ref{thm:8}. Furthermore, for example, set $F=S^k$ for $k \geq 2$, we can construct a desired
 round fold map into the plane on any manifold $M^{\prime}$ regarded
 as the total space of an linear $F$-bundle over $M$ which is orientable. 
\end{Ex}

\end{document}